\documentclass{amsart}
\usepackage{amsfonts,amsmath,amsthm,amssymb}
\newtheorem{thm}{Theorem}
\newtheorem{lem}[thm]{Lemma}

\theoremstyle{remark}
\newtheorem{rem}{Remark}

\theoremstyle{definition}
\newtheorem{dfn}[thm]{Definition}

\newcommand{\norm}[2]{\left\|\left.{#1}\right|{#2}\right\|}

\newcommand{\R}{\mathbb{R}}
\newcommand{\Rn}{{\mathbb{R}^n}}
\newcommand{\N}{\mathbb{N}}
\newcommand{\C}{\mathbb{C}}
\newcommand{\Z}{\mathbb{Z}}
\newcommand{\supp}{{\rm supp\ }}

\newcommand{\ellqp}{{\ell_{q(\cdot)}(L_{p(\cdot)})}}
\newcommand{\p}{{p(\cdot)}}

\newcommand{\q}{{q(\cdot)}}
\newcommand{\Lp}{L_\p(\Rn)}
\newcommand{\Lpq}{L_{\p,\q}(\Rn)}
\renewcommand{\P}{\mathcal{P}(\Rn)}

\newcommand{\esssup}{\operatornamewithlimits{ess\,sup}}
\newcommand{\essinf}{\operatornamewithlimits{ess\,inf}}

\title{Lorentz spaces with variable exponents}
\author{Henning Kempka}
\address{Mathematical Institute, Friedrich-Schiller-University Jena, D--07737 Jena, Germany}
\curraddr{Institute of Applied Analysis, Technical University Bergakademie Freiberg, D--09596 Freiberg, Germany}
\email{henning.kempka@uni-jena.de}
\thanks{The first author acknowledges the financial support by the DFG grants HA 2794/5-1 and KE 1847/1-1.}

\author{Jan Vyb\'\i ral}
\address{Department of Mathematics, Technical University Berlin, Street of 17. June 136, D-10623 Berlin, Germany}
\email{vybiral@math.tu-berlin.de}
\thanks{The second author acknowledges the support by the DFG Research Center MATHEON ``Mathematics for key technologies'' in Berlin.}

\begin{document}
\maketitle
\begin{abstract}
We introduce Lorentz spaces $L_{p(\cdot),q}(\R^n)$ and $L_{p(\cdot),q(\cdot)}(\R^n)$ with variable exponents.
We prove several basic properties of these spaces including embeddings and the identity $L_{p(\cdot),p(\cdot)}(\R^n)=L_{p(\cdot)}(\R^n)$. We also show that these spaces arise through real interpolation between $L_{\p}(\R^n)$ and $L_\infty(\R^n)$.
Furthermore, we answer in a negative way the question posed in \cite{DHN} 
whether the Marcinkiewicz interpolation theorem holds in the frame of Lebesgue spaces with variable integrability.
\end{abstract}
\section{Introduction}

Lorentz spaces were introduced in \cite{Lor1, Lor2} as a generalization of classical Lebesgue spaces and have become a standard tool in mathematical analysis, cf. \cite{AM,CPSS, BS, Grafaneu}.
For an introduction to Lorentz spaces we refer e.g. to \cite[Chapter V]{SW}, \cite[Chapter 4]{BS} or \cite[Chapter 1]{Grafaneu}.

One of the main ingredients of the theory of Lorentz spaces is the celebrated Marcinkiewicz interpolation theorem, which states that under certain conditions one can deduce
the strong boundedness of a sublinear operator $T$ on the interpolation spaces provided that the operator is weakly bounded at the endpoints of the interpolation pair.
This approach was used for example in the classical book of Stein \cite{S} to prove the boundedness of the Hardy-Littlewood maximal operator on $L_p(\R^n)$ for $1<p\le \infty$.

%It is the aim of this paper to combine the theory of Lorentz spaces with the theory of 

Another classical topic we shall touch in our work are the Lebesgue spaces $L_{\p}(\R^n)$ of variable integrability. The study of this class of function spaces
goes back to Orlicz \cite{Orlicz}. After the survey paper of Kov\'a\v{c}ik and R\'akosn\'\i k \cite{KoRa}, there has been an enormous interest in these spaces 
(and in Sobolev spaces $W^1_{\p}(\Omega)$ built on $L_{\p}(\Omega)$) especially in connection with the application 
in modeling of electrorheological fluids \cite{Ruz1}.
Moreover, the spaces $L_{\p}(\R^n)$ possess interesting applications in the theory of PDE's, variational calculus, financial mathematics and image processing. A recent overview of this vastly growing field is given in \cite{DHHR}.
A fundamental breakthrough concerning spaces of variable integrability was the observation that, under certain regularity assumptions on $\p$, the Hardy-Littlewood maximal operator is also bounded on 
$L_{\p}(\R^n)$, see \cite{Max1}. This result has been generalized to wider classes of exponents $\p$ in \cite{CruzUribe03}, \cite{Max3} and \cite{Max4}.
%As mentioned already above, the crucial breakthrough in the theory of Lebesgue spaces of variable order of integration was the proof of the boundedness of the Hardy-Littlewood maximal operator.
Unfortunately, it turned out that the standard proof of Stein \cite{S} for spaces with constant indices breaks down and completely different methods had to be used to achieve this result, see \cite{DHHR}.

The main aim of this paper is to return to this topic and to study the validity of the Marcinkiewicz interpolation theorem in the frame of Lebesgue spaces with variable integrability.
For this reason, we first explore the possibility of extending the definition of Lorentz spaces to the setting of variable integrability exponents.
We show, that there is really a natural way to define the Lorentz spaces $L_{\p,\q}(\Rn)$, which extends the scale of Lebesgue spaces with variable exponents, i.e. $L_{\p,\p}(\Rn)=L_{\p}(\Rn)$
for $\p=\q$. Later on, we study the interpolation properties of this new scale of spaces. A special case (see Remark \ref{rem:4}) of Theorem \ref{thm:interpol} shows that 
$$
(L_{\p}(\R^n),L_\infty(\R^n))_{\theta,q}=L_{\tilde p(\cdot),q}(\R^n)
$$
for $0<\theta<1$ and 
$$
\frac{1}{\tilde p(\cdot)}=\frac{1-\theta}{\p}.
$$
Finally, we discuss the validity of the Marcinkiewicz interpolation theorem in the context of this new scale of function spaces, an open question posed in \cite{DHN}. It turns out that the answer is negative
and we provide a detailed counterexample to this conjecture.

The structure of the paper is as follows. Section 2 collects classical definitions of Lorentz spaces with constant indices and of Lebesgue spaces with variable integrability. Furthermore, the definition of Lorentz spaces
with variable integrability is given. After collecting some basic properties of this new scale of function spaces in Section 3, we study the real interpolation properties of this scale in Section 4.
Section 5 is devoted to Marcinkiewicz interpolation and contains the counterexample to \cite[Question 2.8]{DHN}. Finally, Section 6 collects some possible research directions and open problems.

At the end of this introduction we would like to mention that another definition of Lorentz spaces $\mathcal{L}^{\p,\q}(\Rn)$ with variable exponents was recently given in \cite{EphreKoki} with \cite{KokiSamko} and \cite{IsraKoki}
as forerunners. Their definition works with non-increasing rearrangement and two variable exponents $\p,\q:[0,\infty)\to [1,\infty]$.
Due to this effect, the important and natural identity $\mathcal{L}^{\p,\p}(\Rn)=L_{\p}(\Rn)$ does not hold in this scale of variable Lorentz spaces. This is a consequence of the definition of $L_{\p}(\Rn)$ where the variable exponent $\p$ is defined on $\Rn$ and not on $[0,\infty)$.
We return to this topic in detail in Remark \ref{rem:Samko}.

\section{Old and new definitions}

In this section we collect the very well known definitions of classical Lorentz spaces (Section \ref{def_sec:Lorentz}) and Lebesgue spaces of variable exponents (Section \ref{def_sec:Lebesgue}). Finally, in Section \ref{def_sec:Lorentz_var},
we provide the definition of Lorentz spaces with variable exponents. For simplicity, we start in Definition \ref{dfn2} with the more intuitive spaces $L_{\p,q}(\Rn)$. The Lorentz spaces $L_{\p,\q}(\Rn)$ with both exponents variable
are introduced shortly after in Definition \ref{dfn3}.

\subsection{Lebesgue spaces with variable exponents}\label{def_sec:Lebesgue}
Let us now recall the definition of the variable Lebesgue spaces $L_{\p}(\Rn)$.
A measurable function $p:\Rn\to(0,\infty]$ is called a variable exponent function if 
it is bounded away from zero. For a set $A\subset\Rn$ we denote $p_A^+=\esssup_{x\in A}p(x)$ and $p_A^-=\essinf_{x\in A}p(x)$; we 
use the abbreviations $p^+=p_{\Rn}^+$ and $p^-=p_{\Rn}^-$. The variable exponent Lebesgue space $L_{p(\cdot)}(\Rn)$ consists of 
all measurable functions $f$ such that there exist an $\lambda>0$ such that the modular
\begin{align*}
\varrho_{L_{p(\cdot)}(\Rn)}(f/\lambda)=\int_{\Rn}\varphi_{p(x)}\left(\frac{|f(x)|}{\lambda}\right)dx
\end{align*}
is finite, where
\begin{align}
	\varphi_p(t)=\begin{cases}t^p&\text{if}\ p\in(0,\infty),\\
0&\text{if}\ p=\infty\ \text{and}\ t\le 1,\\\label{Formelphi}
\infty&\text{if}\ p=\infty\ \text{and}\ t>1.
\end{cases}
\end{align}
This definition is nowadays standard and was used also in \cite[Section 2.2]{AlHa10} and \cite[Definition 3.2.1]{DHHR}.

If we define $\R^n_\infty=\{x\in\Rn:p(x)=\infty\}$ and $\R^n_0=\Rn\setminus\R^n_\infty$, then the Luxemburg norm of a function $f\in L_{p(\cdot)}(\Rn)$ is given by
\begin{align*}
    \norm{f}{L_{p(\cdot)}(\Rn)}&=\inf\{\lambda>0:\varrho_{L_{p(\cdot)}(\Rn)}(f/\lambda)\leq1\}\\
    &=\inf\left\{\lambda>0:\int_{\R^n_0}\!\!\!\left(\frac{|f(x)|}{\lambda}\right)^{p(x)}\!\!\!\!\!\!dx\leq1\ \text{and}\ |f(x)|\le \lambda \ \text{for a.e.}\ x\in\R^n_\infty\right\}.
\end{align*}
It constitutes a norm if $p(\cdot)\geq1$, but it is always a quasi-norm if at least $p^->0$. Furthermore, the spaces $\Lp$ are complete, hence they are (quasi-) Banach spaces if $p^->0$, see \cite{KoRa} for details and further properties.
We denote the class of all measurable functions $p:\Rn\to(0,\infty]$ such that $p^->0$ by $\P$ and the corresponding modular is denoted by $\varrho_\p$ instead of $\varrho_{L_{p(\cdot)}(\R^n)}$.

\subsection{Classical Lorentz spaces}\label{def_sec:Lorentz}
Next, we recall the definition of classical Lorentz spaces as it can be found in \cite{BS} or \cite{Grafaneu}.
This definition makes use of the so-called \emph{non-increasing rearrangement} $f^*$ of a function $f$. For a measurable function $f:\Rn\to\C$, we define first the \emph{distribution function} $\mu_f:[0,\infty)\to[0,\infty]$ by
\begin{align*}
\mu_f(s)=\mu\{x\in\Rn:|f(x)|>s\}, \quad s\ge 0,
\end{align*}
where $\mu$ denotes the Lebesgue measure on $\Rn$.

The distribution function $\mu_f$ provides information about the size of $f$ but not about the local behavior of $f$. 
The (generalized) inverse function to the distribution function is called \emph{non-increasing rearrangement} $f^*:[0,\infty)\to[0,\infty]$ and is defined by
\begin{align*}
f^*(t)=\inf\{s>0:\mu_f(s)\leq t\}.
\end{align*}
%with the convention $\inf\emptyset=\infty$. 
Equipped with these tools, we are now ready to give the definition of the classical Lorentz spaces with constant indices.
\begin{dfn}\label{dfn1}
Given a measurable function $f$ on $\Rn$ and real parameters $0<p,q\leq\infty$, we define
\begin{align*}
	\norm{f}{L_{p,q}(\Rn)}=\begin{cases}\displaystyle
		\left(\int_0^\infty\left(t^{1/p}f^*(t)\right)^q\frac{dt}{t}\right)^{1/q},& \text{if }q<\infty\\
		\displaystyle\sup_{t>0}t^{1/p}f^*(t),&\text{if }q=\infty.
	\end{cases}
\end{align*}
The space of all measurable $f:\Rn\to\C$ with $\norm{f}{L_{p,q}(\Rn)}<\infty$ is denoted by $L_{p,q}(\Rn)$. The spaces are complete and they are normable for $1<p<\infty$ and $1\leq q\leq\infty$, see \cite[Theorem 1.4.11 and Exercise 1.4.3]{Grafaneu}. 
\end{dfn}
The use of non-increasing rearrangement makes it rather difficult to extend Definition \ref{dfn1} to variable exponents $\p,\q:\Rn\to(0,\infty]$.
It is very well known that the spaces $\Lp$ are not translation invariant (see Proposition 3.6.1 in \cite{DHHR}) and therefore the membership of $f$ in $\Lp$ cannot be characterized by any condition on $f^*$ only.\\
To avoid this obstacle, we look for an equivalent characterization of Lorentz spaces $L_{p,q}(\Rn)$ which does not make use of the notion of non-increasing rearrangement.
Therefore we calculate for $p,q<\infty$ using Fubini's theorem and the substitution $s^{p/q}:=t$ (cf. Proposition 1.4.9 in \cite{Grafaneu})
\begin{align}
\norm{f}{L_{p,q}(\Rn)}&=\left(\int_0^\infty\left(t^{1/p}f^*(t)\right)^q\frac{dt}{t}\right)^{1/q}=\left(\int_0^\infty\frac pq f^*(s^{p/q})^{q}{ds}\right)^{1/q}\notag\\
&=\left(\frac pq\right)^{1/q}\left(\int_0^\infty\mu\{s\geq0:f^*(s^{p/q})^{q}> t\}dt\right)^{1/q}\notag\\
&=\left(\frac pq\right)^{1/q}\left(\int_0^\infty\mu\{s\geq0:f^*(s^{p/q})> t^{1/q}\}dt\right)^{1/q}\notag\\
&=p^{1/q}\left(\int_0^\infty\lambda^q \mu\{s\geq0:f^*(s^{p/q})> \lambda\}\frac{d\lambda}{\lambda}\right)^{1/q}\notag\\
&=p^{1/q}\left(\int_0^\infty\lambda^q\mu\{s\geq0:f^*(s)> \lambda\}^{q/p}\frac{d\lambda}{\lambda}\right)^{1/q}\notag\\
&=p^{1/q}\left(\int_0^\infty\lambda^q \norm{\chi_{\{x\in\Rn:|f(x)|> \lambda\}}}{L_p(\Rn)}^{q}\frac{d\lambda}{\lambda}\right)^{1/q}\label{Formel1}.	
\end{align}
Here, $\chi_{\{x\in\Rn:|f(x)|> \lambda\}}$ stands for the characteristic function of the set $\{x\in\Rn:|f(x)|> \lambda\}$. If no confusion is possible, this will also be denoted by
$\chi_{\{|f|>\lambda\}}$.

The equation \eqref{Formel1} can be discretized and we derive
\begin{align}
\norm{f}{L_{p,q}(\Rn)}
&\sim p^{1/q}\left(\sum_{k=-\infty}^\infty 2^{kq}\norm{\chi_{\{x\in\Rn:|f(x)|>2^k\}}}{L_p(\Rn)}^q\right)^{1/q}.\label{Formel2}
\end{align}

\subsection{Lorentz spaces with variable exponents}\label{def_sec:Lorentz_var}
The expression \eqref{Formel1} for the norm can be generalized quite easily to variable exponent $\p$ with $q$ constant. Surprisingly enough, even $q$ can be considered variable when we use the spaces 
$\ellqp$ of Almeida and H\"ast\"o \cite{AlHa10} and the discretized equation \eqref{Formel2}.
Furthermore we do not destroy the local properties of the function $f$, since it gets not rearranged and the exponents map from $\Rn$.
%The formula \eqref{Formel1} allows a straightforward generalization to the situation where only $p$ is variable.
\begin{dfn}\label{dfn2}
Let $p\in\P$ be a variable exponent with range $0<p^-\leq p^+\leq\infty$ and let $0<q\leq\infty$. Then $L_{p(\cdot),q}(\R^n)$ is the collection of all measurable functions $f:\Rn\to\C$ such that
\begin{align}\label{LorentzNormqconst}
\norm{f}{L_{\p,q}(\Rn)}=
\begin{cases}\displaystyle
	\left(\int_0^\infty \lambda^q\norm{\chi_{\{x\in\Rn:|f(x)|>\lambda\}}}{\Lp}^q\frac{d\lambda}{\lambda}\right)^{1/q},&\text{ if }q<\infty\\ 
	\displaystyle\sup_{\lambda>0}\lambda\norm{\chi_{\{x\in\Rn:|f(x)|>\lambda\}}}{\Lp},&\text{ if }q=\infty
\end{cases}
\end{align}
is finite.
\end{dfn}
Using the $\ellqp$ spaces introduced recently in \cite{AlHa10}, we may even consider the situation, where also $q$ is variable.
Let us recall their approach. For a sequence $(f_k)_{k\in\Z}$ we define the modular
\begin{align*}
\varrho_{\ell_\q(L_\p)}((f_k)_k)=\sum_{k\in\Z}\inf\left\{\lambda_k>0:\varrho_\p\left(\frac{f_k}{\lambda_k^{1/\q}}\right)\leq1\right\},
\end{align*}
with the convention $\lambda^{1/\infty}=1$. If $q^+<\infty$ or if $\q\leq\p$ we can replace this by a simpler expression
\begin{align}\label{EasyModular}
\varrho_{\ell_\q(L_\p)}((f_k)_k)=\sum_{k\in\Z}\norm{\varphi_\q(|f_k(\cdot)|)}{L_{\frac\p\q}(\Rn)},
\end{align}
which is much more intuitive. Here $\varphi_q(t)$ equals basically $t^q$, see \eqref{Formelphi}. The norm in these spaces gets defined as usual as the Luxemburg norm
\begin{align*}
\norm{(f_k)_k}{\ellqp}=\inf\{\mu>0:\varrho_{\ellqp}\left({f_k}/{\mu}\right)\leq1\}.
\end{align*}
Up to now, it is not completely clear under which conditions on $\p$ and $\q$ the expression above becomes a norm. It was shown in \cite{AlHa10} that it always constitutes a quasi-norm if $p^-,q^->0$. 
Further it is known (see \cite{KV11}) that it is a norm if either $\frac1\p+\frac1\q\leq1$ holds pointwise for all $x\in\Rn$ or if $1\leq\q\leq\p\leq\infty$ holds pointwise. Also in this work there is given 
an example where $\min(\p,\q)\geq1$ but the triangle inequality does not hold. Let us mention that it is an open question if there exists an equivalent norm on $\ellqp$
whenever $\min(\p,\q)\geq1$. Nevertheless, since our exponents are between $(0,\infty]$ we generally work with quasi-norms and there are no obstacles with that.\\
We use now the modular $\varrho_{\ell_\q(L_\p)}$ and \eqref{Formel2} to define the variable Lorentz spaces $\Lpq$.
\begin{dfn}\label{dfn3}
Let $p,q\in\P$ be two variable exponents with range $0<p^-\leq p^+\leq\infty$ and $0<q^-\leq q^+\leq\infty$. Then $\Lpq$ is the collection of all measurable functions $f:\Rn\to\C$ such that
\begin{align}\label{def:norm1}
\norm{f}{\Lpq}=\inf\Bigl\{\lambda>0:\varrho_{\ell_\q(L_\p)}\Bigl(2^k\chi_{\{x\in\Rn:|f(x)/\lambda|>2^k\}}\Bigr)\le 1\Bigr\}<\infty.
%\norm{\left(2^k\chi_{\{x\in\Rn:|f(x)|>2^k\}}\right)_{k=-\infty}^\infty}{\ellqp}<\infty.
\end{align}
\end{dfn}

Before we discuss the properties of these new function spaces we derive an equivalent expression for $\norm{f}{\Lpq}$.
\begin{lem}\label{lem:equiv}Let $p,q\in\P$ be two variable exponents with range $0<p^-\leq p^+\leq\infty$ and $0<q^-\leq q^+\leq\infty$. Then
\begin{equation}\label{def:norm2}
\norm{f}{\Lpq}\approx \norm{\left(2^k\chi_{\{x\in\Rn:|f(x)|>2^k\}}\right)_{k=-\infty}^\infty}{\ellqp}.
\end{equation}
\end{lem}
\begin{proof}
If $\lambda=2^j$ for some $j\in\Z$, we obtain
$$
\varrho_{\ell_\q(L_\p)}\Bigl(2^k\chi_{\{x\in\Rn:|f(x)/\lambda|>2^k\}}\Bigr)=
\varrho_{\ell_\q(L_\p)}\biggl(\frac{2^k\chi_{\{x\in\Rn:|f(x)|>2^k\}}}{\lambda}\biggr).
$$
The rest of the proof then follows by simple monotonicity arguments.
\end{proof}
\begin{rem}
	The somehow more complicated definition of the quasi-norm in Definition \ref{dfn3} was necessary. Only the expression in \eqref{def:norm1} is homogeneous; i.e.~
	\begin{align*}
		\norm{\lambda f}{\Lpq}=|\lambda|\cdot \norm{f}{\Lpq}\quad\text{for all $\lambda\in\R$.}
	\end{align*}  
%	If we would use the right-hand side of \eqref{def:norm2}
%	\begin{align*}
%		\norm{f}{\tilde{L}_{\p,\q}(\Rn)}=\norm{\left(2^k\chi_{\{x\in\Rn:|f(x)|>2^k\}}\right)_{k=-\infty}^\infty}{\ellqp}
%	\end{align*}
%for the definition, then we would get a non-homogeneous quasi-norm as easy examples show. 
Easy examples show that the right-hand side of \eqref{def:norm2} fails to have this property.
Nevertheless due to Lemma \ref{lem:equiv}, both expressions are equivalent and therefore define the same spaces $\Lpq$.
In majority of our considerations, we shall work with the somehow simpler expression \eqref{def:norm2}.
\end{rem}
If $q(\cdot)=q$ is a constant function, then the Proposition 3.3 in \cite{AlHa10} shows that $\ell_{q(\cdot)}(L_{p(\cdot)})$ is really an iterated space, i.e.
$$
 \|(f_k)_{k\in\Z}|\ell_{q}(L_{p(\cdot)})\|=\Bigl(\sum_{k\in\Z}\|f_k|L_{\p}(\R^n)\|^q\Bigr)^{1/q}
$$
(with an appropriate modification if $q=\infty$). By \eqref{Formel2} and \eqref{def:norm2}, we obtain that Definitions \ref{dfn2} and \ref{dfn3} are equivalent.

Moreover, we observe by \eqref{Formel1} and \eqref{Formel2} that for constant functions $p(x)=p$ and $q(x)=q$ we get an equivalent norm for the usual Lorentz spaces $L_{p,q}(\Rn)$, whenever $p,q<\infty$.
A similar calculation justifies this fact also if $p<q=\infty$. If $p=\infty$, then the usual Lorentz spaces $L_{\infty,q}(\Rn)$ defined by Definition \ref{dfn1} consist only of the zero function whenever $0<q<\infty$, see Section 1.4.2 in \cite{Grafaneu}.
It is easy to see, that Definition \ref{dfn2} applied to $p(\cdot)=\infty$ and $q<\infty$ gives $L_{\infty,q}(\Rn)=L_\infty(\Rn)$.
Nevertheless, we will show that $L_{\p,\p}(\Rn)=\Lp$ and therefore the case $p=q=\infty$ is also included for $p=\infty$.\\
Summarizing, our spaces $\Lpq$ are equivalent to the usual Lorentz spaces $L_{p,q}(\Rn)$ if $\p=p$ and $\q=q$ are constant functions. The only exception is the case if $p=\infty$ and $0<q<\infty$, see Theorem \ref{thm:embedding}.

\begin{rem}\label{rem:Samko}
Another approach to generalize this definition to variable exponents was given in \cite{EphreKoki}, with forerunners \cite{KokiSamko} and \cite{IsraKoki}.
They introduced the spaces ${\mathcal L}^{\p,\q}(\Omega)$ by the corresponding quasi-norm
$$
\|f|{\mathcal L}^{\p,\q}(\Omega)\|=\|t^{\frac{1}{p(t)}-\frac{1}{q(t)}}f^*(t)|L_{\q}(0,|\Omega|)\|,
$$
where $\Omega\subset \Rn$ is a measurable set, $|\Omega|$ is its Lebesgue measure and $p,q:(0,|\Omega|)\to (0,\infty)$ are variable exponents.
The spaces ${\mathcal L}^{\p,\q}(\Omega)$ coincide with usual Lorentz spaces $L_{p,q}(\Omega)$ if $\p=p$ and $\q=q$ are constant.
On the other hand, in this scale there is no hope for the identity ${\mathcal L}^{\p,\p}(\Omega)=L_\p(\Omega)$ to hold, since in the definition of the Lebesgue spaces the variable exponent $\p$ is defined on whole $\Omega$.
\end{rem}

\section{Basic properties}
In this section, we prove several basic properties of the new scale of function spaces. 

\begin{thm} Let $p,q\in \P$. Then $\Lpq$ are quasi-Banach spaces.
\end{thm}
\begin{proof}
To prove that \eqref{def:norm1} defines a quasi-norm, we only have to show the quasi-triangle inequality. We use the estimate
\begin{align*}
\{x\in\Rn:|f(x)+g(x)|>2^k\}&\subset\{x\in\Rn:|f(x)|+|g(x)|>2^k\}\\
%&\subset\{x\in\Rn:|f(x)|>2^{k-1}\text{ or }|g(x)|>2^{k-1}\}\\
&\subset\{x\in\Rn:|f(x)|>2^{k-1}\}\cup\{x\in\Rn:|g(x)|>2^{k-1}\}
\end{align*}
to obtain
\begin{align*}
\chi_{\{x\in\Rn:|f(x)+g(x)|>2^k\}}&\leq \chi_{\{x\in\Rn:|f(x)|>2^{k-1}\}}+\chi_{\{x\in\Rn:|g(x)|>2^{k-1}\}}.
\end{align*}
This in turn implies
\begin{align*}
&\norm{f+g}{\Lpq}\approx\norm{\left(2^k\chi_{\{x\in\Rn:|f(x)+g(x)|> 2^k\}}\right)_{k=-\infty}^\infty}{\ellqp}\\
&\qquad\le \norm{\left(2^k\chi_{\{x\in\Rn:|f(x)|>2^{k-1}\}}\right)_{k=-\infty}^\infty+\left(2^k\chi_{\{x\in\Rn:|g(x)|>2^{k-1}\}}\right)_{k=-\infty}^\infty}{\ellqp}\\
&\qquad\le c\biggl\{  \norm{\left(2^k\chi_{\{x\in\Rn:|f(x)|>2^{k-1}\}}\right)_{k=-\infty}^\infty}{\ellqp}+\norm{\left(2^k\chi_{\{x\in\Rn:|g(x)|>2^{k-1}\}}\right)_{k=-\infty}^\infty}{\ellqp}\biggr\}\\
&\qquad\lesssim \norm{f}{\Lpq}+\norm{g}{\Lpq},
\end{align*}
where $c$ is the constant from the quasi-triangle inequality of $\ellqp$.

To show that the spaces $\Lpq$ are complete we take a Cauchy sequence $(f_l)_{l\in\N}\subset\Lpq$.
%From the embedding $\Lpq\hookrightarrow L_{\p,\infty}(\Rn)$, see Theorem \ref{thm:embedding}, we get that $(f_m)_{m\in\N}$ is a Cauchy sequence in $L_{\p,\infty}(\Rn)$ as well.
We chose a subsequence (which we denote by $(f_l)_{l\in\N}$ again) with
$$
\|f_{l+1}-f_l|L_{\p,\q}\|\le \frac{1}{2^{2l}},\quad l\in\N.
$$
For notational reasons, we put $f_0=0.$
We consider the function
$$
g(t):=\sum_{l=0}^\infty |f_{l+1}(t)-f_l(t)|.
$$
As $\chi_{\{g>\lambda\}}(x)\le \sum_{l=0}^\infty \chi_{\{|f_{l+1}-f_l|>\lambda/2^{l+1}\}}(x)$, we obtain
\begin{align*}
\|\chi_{\{g>\lambda\}}|\Lp\|^r&\le \sum_{l=0}^\infty \|\chi_{\{|f_{l+1}-f_l|>\lambda/2^{l+1}\}}|\Lp\|^r\\
&\le \sum_{l=0}^\infty \frac{2^{(l+1)r}}{\lambda^r}\cdot \|f_{l+1}-f_l|L_{\p,\infty}\|^r
\lesssim \sum_{l=0}^\infty \frac{2^{(l+1)r}}{\lambda^r} 2^{-2lr}
\end{align*}
where $r=\min(p^-,1)$ and we have used the embedding $\Lpq\hookrightarrow L_{\p,\infty}(\Rn)$, see Theorem \ref{thm:embedding}.
As the last sum converges, we get $\|\chi_{\{g>\lambda\}}|\Lp\|\to 0$ for $\lambda\to \infty$ and $g$ is finite almost everywhere.
Therefore, the series
$$
f(x)=\sum_{l=0}^\infty f_{l+1}(x)-f_l(x)\quad \text{and}\quad \tilde f(x)=\sum_{l=1}^\infty f_{l+1}(x)-f_l(x)=f(x)-f_1(x),  \quad x\in\R^n
$$
converge also almost everywhere.

It remains to show that $f\in\Lpq$ and $f_l\to f$ in $\Lpq.$

The estimate $\displaystyle 2^k \chi_{\{|\tilde f|>2^k\}}\le \sum_{l=1}^\infty 2^k\chi_{\{|f_{l+1}-f_l|>2^{k-l}\}}$ implies that
\begin{align*}
\|2^k \chi_{\{|\tilde f|>2^k\}}|\ellqp\|^\varrho&\lesssim \sum_{l=1}^\infty \|2^k\chi_{\{|f_{l+1}-f_l|>2^{k-l}\}}|\ellqp\|^\varrho\\
&= \sum_{l=1}^\infty 2^{l\varrho}\|2^{k-l}\chi_{\{|f_{l+1}-f_l|>2^{k-l}\}}|\ellqp\|^\varrho\lesssim\sum_{l=1}^\infty 2^{l\varrho}\cdot 2^{-2l\varrho}<\infty,
\end{align*}
where $\varrho>0$ is chosen small enough, cf. \cite[Theorem 3.8]{AlHa10}.%\fix{Maybe ...+Aoki Rolewicz}

Therefore, $\tilde f\in\Lpq$ and $f=\tilde f+f_1\in \Lpq.$

Finally, for $l\in\N$ fixed, we consider
$$
f-f_l=\sum_{m=l}^\infty (f_{m+1}-f_m).
$$
The estimate $\chi_{\{|f-f_l|>2^k\}}\le \sum_{m=l}^\infty \chi_{\{|f_{m+1}-f_m|>2^{k-(m-l+1)}\}}$ implies the convergence $\|f-f_l|\Lpq\| \to0$ for $l\to\infty$ in a similar manner as above.
\end{proof}
We continue with a theorem showing that the scale of variable Lorentz spaces
includes the scale of Lebesgue spaces with variable exponent. We would like to emphasize, that the identity $L_{\p,\p}(\Rn)=\Lp$ holds without any restrictions on $\p$ with $0<p^-\leq p^+\leq\infty$.
\begin{thm}
If $p\in\P$, then it holds $L_{\p,\p}(\Rn)=\Lp$.
\end{thm}
\begin{proof}
We want to show
\begin{align}\label{ModulareInequality}
\varrho_\p(f/2)\leq\varrho_{\ell_\p(L_\p)}\left((2^k\chi_{\{x\in\R^n:|f(x)|>2^k\}})_{k=-\infty}^\infty\right)\leq\varrho_\p(cf),
\end{align}
where $c=(1-2^{-p^-})^{-1/p^-}$. From the inequalities above we conclude easily
\begin{align*}
	\frac12\norm{f}{\Lp}\lesssim\norm{f}{L_{\p,\p}(\Rn)}\lesssim c\norm{f}{\Lp},
\end{align*}
which proves the theorem. We first treat the case $|\Rn_{\!\!\!\infty}|=|\{x\in\Rn:p(x)=\infty\}|=0$. At the end we comment on the case $|\Rn_{\!\!\!\infty}|>0$. 
Since $\p=\q$ we can use the easy expression \eqref{EasyModular} for the modular and then the first inequality in \eqref{ModulareInequality} follows from
\begin{align*}
	\varrho_{\ell_\p(L_\p)}\left((2^k\chi_{\{|f|>2^k\}})_{k=-\infty}^\infty\right)&=\sum_{k=-\infty}^\infty\norm{|2^k\chi_{\{|f|>2^k\}}|^{p(x)}}{L_1(\Rn)}\\
	&=\int_\Rn\sum_{\{k\in \Z:2^k<|f(x)|\}}2^{kp(x)}dx.
\end{align*}
For fixed $x\in\Rn$ with $|f(x)|>0$ we choose the unique $k_x\in\Z$ with $2^{k_xp(x)}<|f(x)|^{p(x)}\leq2^{(k_x+1)p(x)}$ and obtain
\begin{align}\label{fkZero}
	\sum_{\{k\in\Z:2^k<|f(x)|\}}2^{kp(x)}=\sum_{k=-\infty}^{k_x}\left(2^{-p(x)}\right)^{-k}=2^{k_xp(x)}\frac{1}{1-2^{-p(x)}}.
\end{align}
Using $1\leq\frac{1}{1-2^{-p(x)}}$ we get from \eqref{fkZero}
\begin{align*}
	\varrho_{\ell_\p(L_\p)}\left((2^k\chi_{\{|f(x)|>2^k\}})_{k=-\infty}^\infty\right)&=\int_\Rn\sum_{\{k\in\Z:2^k<|f(x)|\}}2^{kp(x)}dx\\
	&\geq\int_\Rn2^{k_xp(x)}dx=\int_\Rn2^{(k_x+1)p(x)}2^{-p(x)}dx\\
	&\ge\int_\Rn|\frac12f(x)|^{p(x)}dx=\varrho_\p(f/2).
\end{align*}
The converse inequality uses again \eqref{fkZero} with $\frac{1}{1-2^{-p(x)}}\leq\frac{1}{1-2^{-p^-}}$ and follows in a similar way
\begin{align*}
	\varrho_{\ell_\p(L_\p)}\left((2^k\chi_{\{|f(x)|>2^k\}})_{k=-\infty}^\infty\right)&=\int_\Rn\sum_{\{k\in\Z:2^k<|f(x)|\}}2^{kp(x)}dx\\
	&\hspace{-5em}=\int_\Rn2^{k_xp(x)}\frac{1}{1-2^{-p(x)}}dx\leq\int_\Rn2^{k_xp(x)}\frac{1}{1-2^{-p^-}}dx\\
	&\hspace{-5em}\leq\int_\Rn|f(x)|^{p(x)}\left(\frac{1}{1-2^{-p^-}}\right)^{\frac{p(x)}{p^-}}dx=\varrho_\p(cf),
\end{align*} 
with $c=(1-2^{-p^-})^{-1/p^-}$.

Now, we come back to the case, where $|\Rn_{\!\!\!\infty}|>0$. First, we split our function $f=f_0+f_\infty:=f\cdot\chi_{{\Rn_{0}}}+f\cdot\chi_{\Rn_{\!\!\!\infty}}$. Then we use the considerations above and
\begin{align*}
	\norm{f_0}{\Lp}+\norm{f_\infty}{\Lp}&\leq 2\norm{f_0+f_\infty}{\Lp},
	\intertext{see \cite[Remark 3.2.3]{DHHR}, and}
	\norm{f_0}{L_{\p,\p}(\Rn)}+\norm{f_\infty}{L_{\p,\p}(\Rn)}&\leq 2\norm{f_0+f_\infty}{L_{\p,\p}(\Rn)},
\end{align*}
	which is implied by
\begin{align*}
	\varrho_{\ell_\p(L_\p)}((2^k\chi_{\{|f_0+f_\infty|>2^k\}})_{k=-\infty}^\infty)&=\varrho_{\ell_\p(L_\p)}((2^k\chi_{\{|f_0|>2^k\}})_{k=-\infty}^\infty)\\
																					&+\varrho_{\ell_\p(L_\p)}((2^k\chi_{\{|f_\infty|>2^k\}})_{k=-\infty}^\infty).
\end{align*} 
\end{proof}
Next we show that the new scale of function spaces satisfies some elementary embeddings, which are very well know in the case of constant exponents.
\begin{thm}\label{thm:embedding}
%Let $0<p^-\leq p^+\leq\infty$, $0<q^-\leq q^+\leq\infty$ and $0<q_0^-\leq q_0^+\leq\infty$, $0<q_1^-\leq q_1^+\leq\infty$.
\begin{itemize}
\item[(i)] Let $p,q_0,q_1\in \P$ with $q_0(\cdot)\leq q_1(\cdot)$ pointwise. Then $L_{\p,q_0(\cdot)}(\Rn)\hookrightarrow L_{\p,q_1(\cdot)}(\Rn)$.
\item[(ii)] Let $q\in\P$. Then $L_{\infty,q(\cdot)}(\R^n)=L_{\infty}(\R^n)$.
\item[(iii)] Let $p_0,p_1,q_0,q_1\in \P$ with $p_0^+<\infty$ and $\alpha:=(p_1/p_0)^->1$. Then the inequality 
\begin{equation}\label{eq:bddsupp1}
\|f|L_{p_0(\cdot),q_0(\cdot)}\| \le c \|f|L_{p_1(\cdot),q_1(\cdot)}\|
\end{equation}
holds for all measurable $f$ with $\supp f\subset [0,1]^n$ with $c$ independent of $f$.
\item[(iv)] Let $p_0,p_1,q\in\P$ with $p_0(\cdot)\le p_1(\cdot)$ pointwise. Then the inequality 
\begin{equation}\label{eq:bddsupp3}
\|f|L_{p_0(\cdot),q(\cdot)}\| \le c \|f|L_{p_1(\cdot),q(\cdot)}\|
\end{equation}
holds for all measurable $f$ with $\supp f\subset [0,1]^n$ with $c$ independent of $f$.
\end{itemize}
\end{thm}
\begin{proof}
The first statement follows from Lemma \ref{lem:equiv} and the embedding $\ell_{q_0(\cdot)}(L_\p)\hookrightarrow\ell_{q_1(\cdot)}(L_\p)$ which has been proven in \cite{AlHa10}.\\
To prove the second part of the theorem, it is enough to use the above embedding in the form
\begin{align*}
L_{\infty,q^-}(\Rn)\hookrightarrow L_{\infty,\q}(\Rn)\hookrightarrow L_{\infty,\infty}(\Rn)=L_\infty(\Rn)
\end{align*}
and the simple embedding $L_\infty(\Rn)\hookrightarrow L_{\infty,q^-}(\Rn)$, which follows directly from Definition \ref{dfn3} and Lemma \ref{lem:equiv}.\\
The proof of the third statement is based on the following simple fact. For every $A\subset\R^n$ with $\mu(A)\le 1$ the following inequality holds
\begin{equation}\label{eq:bddsupp2}
\|\chi_A|L_{p_0(\cdot)}\|\le \|\chi_A|L_{p_1(\cdot)}\|^\alpha,
\end{equation}
where again $\alpha=(p_1/p_0)^->1.$ 

To show \eqref{eq:bddsupp1}, it is enough to assume that $q_1(\cdot)=\infty$ and 
\begin{equation*}
\|f|L_{p_1(\cdot),\infty}(\R^n)\|\approx \sup_{k\in\Z} 2^k\|\chi_{\{|f|>2^k\}}|L_{p_1(\cdot)}(\R^n)\|=1.
\end{equation*}
Using \eqref{eq:bddsupp2}, we obtain
\begin{align*}
\|f|L_{p_0(\cdot),q_0(\cdot)}\|^{q_0^-}&\lesssim \|f|L_{p_0(\cdot),q_0^-}\|^{q_0^-}\le \sum_{k=-\infty}^\infty 2^{kq_0^-}\|\chi_{\{|f|>2^k\}}|L_{p_0(\cdot)}(\R^n)\|^{q_0^-}\\
&\le \sum_{k=-\infty}^0 2^{kq_0^-}+\sum_{k=1}^\infty 2^{kq_0^-}\|\chi_{\{|f|>2^k\}}|L_{p_1(\cdot)}(\R^n)\|^{\alpha q_0^-}\\
&\le c+\sum_{k=1}^\infty 2^{kq_0^-}(2^{-k}\|f|L_{p_1(\cdot),\infty}(\R^n)\|)^{\alpha q_0^-}\le c'
\end{align*}
with an obvious modification if $q_0^-=\infty.$ This justifies \eqref{eq:bddsupp1}.

The proof of the forth statement follows a similar pattern. We start with $f$ such that 
$$
\|f|L_{p_1(\cdot),q(\cdot)}(\R^n)\| \approx \norm{\left(2^k\chi_{\{x\in\Rn:|f(x)|>2^k\}}\right)_{k=-\infty}^\infty}{\ell_{q(\cdot)}(L_{p_1(\cdot)})}=1.
$$
We use again the splitting into two parts, namely with $k\le 0$ and $k\ge 1$, respectively. In the first case, we use the bounded support of $f$ to obtain.
\begin{align*}
\|(2^k\chi_{\{|f|>2^k\}})_{k=-\infty}^0|\ell_{q(\cdot)}(L_{p_0(\cdot)})\|&\le \|(2^k\chi_{[0,1]^n})_{k=-\infty}^0|\ell_{q(\cdot)}(L_{p_0(\cdot)})\|\\
&\lesssim \|(2^k\chi_{[0,1]^n})_{k=-\infty}^0|\ell_{q^-}(L_{p_0(\cdot)})\|\lesssim 1
\end{align*}
The second part with $k\in \N$ may be estimated directly as
\begin{align*}
\|(2^k\chi_{\{|f|>2^k\}})_{k=1}^\infty|\ell_{q(\cdot)}(L_{p_0(\cdot)})\|\le
\|(2^k\chi_{\{|f|>2^k\}})_{k=1}^\infty|\ell_{q(\cdot)}(L_{p_1(\cdot)})\|\lesssim  1,
\end{align*}
which finishes the proof of \eqref{eq:bddsupp3}.
\end{proof}

\begin{rem}
The second part of this theorem is in contrast to \cite[Section 1.4.2]{Grafaneu}, where $L_{\infty,q}(\Rn)=\{0\}$ is stated.
But this is also not surprising since we did not take the extra factor $p^{1/q}$ appearing in \eqref{Formel1} and \eqref{Formel2} into our Definitions \ref{dfn2} and \ref{dfn3} of our variable Lorentz spaces.
\end{rem}

\section{Interpolation}
We stated already in the introduction that the main importance of Lorentz spaces lies in their connection with (real) interpolation theory. In this section, we shall explore the interpolation properties
of Lorentz spaces with variable exponents. But before we come to this, we recall some basics of the interpolation theory, as they may be found for example in the classical monographs \cite{BerghL} and \cite{Triebel}.

We shall touch only the two most important interpolation methods - the real interpolation and the complex interpolation. Complex interpolation of variable exponent spaces has already been treated in \cite{DHN}.
It turned out that the expected result
\begin{align*}
[L_{p_0(\cdot)}(\Rn),L_{p_1(\cdot)}(\Rn)]_{\theta}=L_{p_\theta(\cdot)}(\Rn)
\end{align*} 
does hold for all $0<\theta<1$ and all $\frac1{p_\theta(\cdot)}=\frac{1-\theta}{p_0(\cdot)}+\frac{\theta}{p_1(\cdot)}$ with $p_0^-,p_1^-\geq1$.\\
This complex interpolation result has been complemented in \cite{Kop09} by showing
\begin{align*}
	[L_\p(\Rn), BMO(\Rn)]_\theta=L_{p_\theta(\cdot)}(\Rn)\quad\text{and}\quad[L_\p(\Rn), H_1(\Rn)]_\theta=L_{p_\theta(\cdot)}(\Rn)
\end{align*}
under some regularity conditions on $\p$. 

We shall therefore concentrate on the real interpolation method (the so-called \emph{K-method}). Let $X_0$ and $X_1$ be two (quasi-)Banach spaces, which are both embedded into a topological vector space $Y$.
Then the spaces $X_0+X_1$ is defined as the set of all $x\in Y$, which may be written as $x=x_0+x_1$ with $x_0\in X_0$ and $x_1\in X_1.$

For any $x\in X_0+X_1$ and any $0<t<\infty$, the so-called \emph{Peetre K-functional} is defined by
\begin{equation}\label{eq:K}
K(x,t,X_0,X_1)=\inf\{\|x_0|X_0\|+t\|x_1|X_1\|:x=x_0+x_1, x_0\in X_0, x_1\in X_1\}.
\end{equation}
If the spaces $X_0$ and $X_1$ are fixed and no confusion is possible, then we abbreviate this to $K(x,t)$. If $0<\theta<1$ and $0<q\le\infty$, then the \emph{real interpolation space} $(X_0,X_1)_{\theta,q}$ is defined
as the set of all $x\in X_0+X_1$, such that
\begin{align*}
\|x|(X_0,X_1)_{\theta,q}\|=\begin{cases} \displaystyle \biggl(\int_0^\infty t^{-\theta q} K(x,t)^q \frac{dt}{t}\biggr)^{1/q},\quad &\text{if\ }q<\infty,\\
\displaystyle \sup_{t>0}t^{-\theta}K(x,t),&\text{if\ }q=\infty
\end{cases}
\end{align*}
is finite. 

\begin{thm}\label{thm:interpol} Let $p,q_0\in\P$ with $p^+<\infty$. Let $0<q\le\infty$ and $0<\theta<1$ and put
$$
\frac{1}{\tilde p(\cdot)}=\frac{1-\theta}{p(\cdot)}.
$$
Then
\begin{equation}\label{eq:interpol1}
(L_{\p,q_0(\cdot)}(\R^n),L_\infty(\R^n))_{\theta,q}=L_{\tilde p(\cdot),q}(\R^n)
\end{equation}
in the sense of equivalent quasi-norms.
\end{thm}
\begin{proof}
To prove \eqref{eq:interpol1}, we will justify the following chain of embeddings.
\begin{align}
\notag L_{\tilde p(\cdot),q}(\R^n)&\hookrightarrow
(L_{\p,q_0^-}(\R^n),L_\infty(\R^n))_{\theta,q}\hookrightarrow
(L_{\p,q_0(\cdot)}(\R^n),L_\infty(\R^n))_{\theta,q}\\
\label{eq:review1}&\hookrightarrow
(L_{\p,\infty}(\R^n),L_\infty(\R^n))_{\theta,q}\hookrightarrow
L_{\tilde p(\cdot),q}(\R^n).
\end{align}
The second and third embedding in \eqref{eq:review1} follow by monotonicity, cf. Theorem \ref{thm:embedding}.
The last embedding in \eqref{eq:review1}, namely
\begin{equation}\label{eq:interpol2}
(L_{\p,\infty}(\R^n),L_\infty(\R^n))_{\theta,q}\hookrightarrow L_{\tilde p(\cdot),q}(\R^n),
\end{equation}
will be proven in Step 1.
%It suffices to prove
%\begin{align}
%	(L_{\p,\infty}(\R^n),L_\infty(\R^n))_{\theta,q}&\hookrightarrow L_{\tilde p(\cdot),q}(\R^n)\label{Thm8:Step1}
%	\intertext{and}
%	L_{\tilde p(\cdot),q}(\R^n)&\hookrightarrow (L_{\p}(\R^n),L_\infty(\R^n))_{\theta,q}.\label{Thm8:Step2}	
%\end{align}
%From \eqref{Thm8:Step1} we easily get one direction of the statement \eqref{eq:interpol1} by monotonicity and Theorem \ref{thm:embedding}. 
% The other  direction follows from \eqref{Thm8:Step2} and the fact that we only work with norms of characteristic functions in Step 2.
Finally, in Step 2 we shall prove that
$$
L_{\tilde p(\cdot),q}(\R^n)\hookrightarrow
(L_{\p}(\R^n),L_\infty(\R^n))_{\theta,q}.
$$
The proof of this embedding only works with norms of characteristic functions, which do not depend on the second parameter of the Lorentz space.
This is very well known for classical Lorentz spaces, follows from Definition \ref{dfn2} for Lorentz spaces with variable $p(\cdot)$ and $q$ constant,
and finally follows by monotonicity also for Lorentz spaces with both indices variable. Therefore, the proof given in Step 2 also justifies the first embedding in \eqref{eq:review1}.

\emph{Step 1.} We shall prove \eqref{eq:interpol2}, i.e. that
\begin{equation}\label{eq:interpol3}
\int_0^\infty \lambda^q\norm{\chi_{\{x\in\Rn:|f(x)|> \lambda\}}}{L_{\tilde p(\cdot)}(\R^n)}^q\frac{d\lambda}{\lambda}\lesssim
\int_0^\infty t^{-\theta q} K(f,t)^q \frac{dt}{t}.
\end{equation}
We shall use that
\begin{align*}
K(f,t)&=\inf\{\|f^0\|_{\p,\infty}+t\|f^1\|_\infty:f=f^0+f^1\}\\
&=\inf_{\mu>0}\{\|(|f(x)|-\mu)_+\|_{p(\cdot),\infty}+t\|\min(|f(x)|,\mu)\|_\infty\}\\
&\ge \inf_{\mu>0}\{\mu \|\chi_{\{x\in\Rn:|f(x)|\ge 2\mu\}}\|_{p(\cdot)}+t\|\min(f(x),\mu)\|_\infty\}
\end{align*}
for every fixed $t>0$.

We denote $h(\lambda)=\|\chi_{\{x\in\Rn:|f(x)|\ge \lambda\}}\|_{\p}$ for $\lambda>0$ and $f_*(t)=\sup\{\lambda>0:h(\lambda)\ge t\}$ its generalized inverse function.
Using the assumption $p^+<\infty$, we obtain that $h(f_*(t))\ge t$ for all $t>0$.
Then we choose $\mu$ by $\mu=f_*(t)/2$. This leads to $K(f,t)\ge f_*(t)h(f_*(t))/2\ge tf_*(t)/2.$ 
The proof is then a consequence of the following two estimates on the left and right hand side of \eqref{eq:interpol3}
\begin{align*}
LHS\eqref{eq:interpol3}&=\int_0^\infty \lambda^q h(\lambda)^{(1-\theta)q}\frac{d\lambda}{\lambda}\approx \sum_{k=-\infty}^\infty \int_{\lambda:2^k< h(\lambda)\le2^{k+1}}\lambda^q h(\lambda)^{(1-\theta)q}\frac{d\lambda}{\lambda}\\
&\lesssim \sum_{k=-\infty}^\infty 2^{k(1-\theta)q}\int_{\lambda:2^k\le h(\lambda)}\lambda^q \frac{d\lambda}{\lambda}
\le \sum_{k=-\infty}^\infty 2^{k(1-\theta)q}\int_{0}^{f_*(2^k)}\lambda^q \frac{d\lambda}{\lambda}\lesssim \sum_{k=-\infty}^\infty 2^{k(1-\theta)q}f_*(2^k)^q
\end{align*}
and
\begin{align*}
RHS\eqref{eq:interpol3}&\ge \sum_{k=-\infty}^\infty \int_{2^k}^{2^{k+1}}t^{(1-\theta)q}f_*(t)^q\frac{dt}{t}\gtrsim \sum_{k=-\infty}^\infty 2^{k(1-\theta)q}\int_{2^k}^{2^{k+1}}f_*(t)^q\frac{dt}{t}\\
&\gtrsim\sum_{k=-\infty}^\infty 2^{(k+1)(1-\theta)q}f_*(2^{k+1})^q =\sum_{k=-\infty}^\infty 2^{k(1-\theta)q}f_*(2^{k})^q.
\end{align*}
If $q=\infty$, only notational modifications are necessary.

\emph{Step 2.} Next, we prove that
\begin{equation}\label{eq:interpol10}
L_{\tilde p(\cdot),q}(\R^n)\hookrightarrow (L_{\p}(\R^n),L_\infty(\R^n))_{\theta,q},
\end{equation}
i.e.
\begin{equation}\label{eq:interpol11}
\int_0^\infty t^{-\theta q} K(f,t)^q \frac{dt}{t}\lesssim\int_0^\infty \lambda^q\norm{\chi_{\{x\in\Rn:|f(x)|> \lambda\}}}{L_{\tilde p(\cdot)}(\R^n)}^q\frac{d\lambda}{\lambda}.
\end{equation}
We start again with a reformulation of $K(f,t)$.
\begin{align*}
K(f,t)&=\inf\{\|f^0\|_{\p}+t\|f^1\|_\infty:f=f^0+f^1\}\\
&=\inf_{\mu>0}\{\|(|f(x)|-\mu)_+\|_{p(\cdot)}+t\|\min(|f(x)|,\mu)\|_\infty\}\\
&\le \inf_{\mu>0}\{\|f\chi_{\{x\in\Rn:|f(x)|> \mu\}}\|_{p(\cdot)}+t\mu\}\\
&\lesssim \inf_{\mu>0}\{\|\sum_{j=0}^\infty 2^j\mu\chi_{\{x\in\Rn:|f(x)|> 2^j\mu\}}\|_{p(\cdot)}+t\mu\}\\
&\le \inf_{\mu>0}\{\sum_{j=0}^\infty 2^j\mu\|\chi_{\{x\in\Rn:|f(x)|> 2^j\mu\}}\|_{p(\cdot)}+t\mu\}\\
&= \inf_{\mu>0}\{\sum_{j=0}^\infty 2^j\mu h(2^j\mu)+t\mu\},
\end{align*}
where we denoted $h(\lambda):=\|\chi_{\{x\in\Rn:|f(x)|> \lambda\}}\|_{p(\cdot)}$. Let us remark, that we have assumed $p^-\ge 1$ in the calculation above to be able to use the triangle inequality without additional powers.
The modification in the case $p^-<1$ is straightforward and left to the reader.

For fixed $t>0$, we choose $\mu=\mu(t)$ by 
$$
\mu(t):=\inf\{\mu>0:\sum_{j=0}^\infty 2^j h(2^j\mu)\le t\}.
$$
As the function $h$ is right-continuous, we obtain immediately $\sum_{j=0}^\infty 2^j h(2^j\mu(t))\le t$.
We first estimate the right-hand side of \eqref{eq:interpol11} as
\begin{align*}
RHS\eqref{eq:interpol11} \gtrsim \sum_{k=-\infty}^\infty 2^{kq}h(2^k)^{(1-\theta)q}.
\end{align*}
Furthermore, we discretize the left-hand side of \eqref{eq:interpol11} as
\begin{align*}
LHS\eqref{eq:interpol11}&\le\int_0^\infty t^{-\theta q} t^q\mu(t)^q \frac{dt}{t}
=\sum_{k=-\infty}^\infty 2^{kq}\int_{t:2^k< \mu(t)\le 2^{k+1}}t^{(1-\theta)q}\frac{dt}{t}\\
&\le \sum_{k=-\infty}^\infty 2^{kq}\int_{t:2^k< \mu(t)}t^{(1-\theta)q}\frac{dt}{t}.
\end{align*}
If $\mu(t)>2^k$, we obtain
$$
t\le \sum_{j=0}^\infty 2^j h(2^{j+k}).
$$
Therefore, we continue
\begin{align*}
LHS\eqref{eq:interpol11}&\lesssim \sum_{k=-\infty}^\infty 2^{kq}\int_{0}^{\sum_{j=0}^\infty 2^j h(2^{j+k})}t^{(1-\theta)q}\frac{dt}{t}
\lesssim \sum_{k=-\infty}^\infty 2^{kq} \left(\sum_{j=0}^\infty 2^j h(2^{j+k})\right)^{(1-\theta)q}.
\end{align*}
If $(1-\theta)q\le 1$, we may write ($l=j+k$)
\begin{align*}
LHS\eqref{eq:interpol11}&\lesssim \sum_{k=-\infty}^\infty 2^{kq} \sum_{j=0}^\infty 2^{j(1-\theta)q} h(2^{j+k})^{(1-\theta)q}\\
&=\sum_{l=-\infty}^\infty \sum_{j=0}^\infty 2^{(l-j)q} 2^{j(1-\theta)q}h(2^l)^{(1-\theta)q}\\
&=\sum_{l=-\infty}^\infty 2^{lq}h(2^l)^{(1-\theta)q}\sum_{j=0}^\infty 2^{-j\theta q}\lesssim \sum_{l=-\infty}^\infty 2^{lq}h(2^l)^{(1-\theta)q}.
\end{align*}
If $(1-\theta)q>1$, we use a similar approach combined with H\"older's inequality
\begin{align*}
LHS\eqref{eq:interpol11}&\lesssim \sum_{k=-\infty}^\infty 2^{kq} \sum_{j=0}^\infty 2^{j(1+\varepsilon)(1-\theta)q} h(2^{j+k})^{(1-\theta)q}\\
&=\sum_{l=-\infty}^\infty \sum_{j=0}^\infty 2^{(l-j)q} 2^{j(1+\varepsilon)(1-\theta)q}h(2^l)^{(1-\theta)q}\\
&=\sum_{l=-\infty}^\infty 2^{lq}h(2^l)^{(1-\theta)q}\sum_{j=0}^\infty 2^{jq(-1 + (1+\varepsilon)(1-\theta))}\lesssim \sum_{l=-\infty}^\infty 2^{lq}h(2^l)^{(1-\theta)q},
\end{align*}
where we have assumed that $\varepsilon>0$ was small enough to obtain $-1 + (1+\varepsilon)(1-\theta)=-\theta+\varepsilon(1-\theta)<0$.
This finishes the proof.
\end{proof}
\begin{rem}\label{rem:4}
\begin{itemize} \item[(i)] 
Let us point out that the assumption $p^+<\infty$ is forced mainly be the technique of generalized inverse functions used in the proof of Theorem \ref{thm:interpol}.
We leave it as an open problem if this assumption might be removed.
\item[(ii)] Of course, taking $q_0(\cdot)=p(\cdot)$ in Theorem \ref{thm:interpol} leads immediately to the following special case
\begin{equation*}
(L_{\p}(\R^n),L_\infty(\R^n))_{\theta,q}=L_{\tilde p(\cdot),q}(\R^n).
\end{equation*}
Therefore the spaces $L_{\tilde p(\cdot),q}(\R^n)$ from Definition \ref{dfn2} naturally arise by real interpolation between $L_{\p}(\R^n)$ and $L_\infty(\R^n)$.
%\item[(ii)] As the proof of Theorem \ref{thm:interpol} works only with norms of characteristic functions, one can easily generalize its statement to
%$$
%(L_{\p,q_0}(\R^n),L_\infty(\R^n))_{\theta,q}=L_{\tilde p(\cdot),q}(\R^n)
%$$
%for any $0<q_0\le \infty$. And due to monotonicity, the result holds for $q_0\in\P$ as well.
\end{itemize}
\end{rem}

\section{Marcinkiewicz interpolation theorem}\label{sec:Marcinkiewicz}

Let $T$ be an operator defined on measurable functions on $\R^n$. We say, that $T$ is sublinear, if
$$
|T(f+g)(x)|\le |Tf(x)|+|Tg(x)|
$$
holds for (almost) every $x\in\R^n$. One of the most important tools in analysis of (sub-)linear operators is the Marcinkiewicz interpolation theorem. Let us recall its statement
as it may be found for example in \cite[Corollary 1.4.21]{Grafaneu}.

\begin{thm}\label{Marcin} Let $\Omega\subset\R^n$ be a measurable set and let $T$ be a sublinear operator, which maps $L_{p_0}(\Omega)$ to $L_{q_0,\infty}(\Omega)$ and $L_{p_1}(\Omega)$ to $L_{q_1,\infty}(\Omega)$, where
$0<p_0\neq p_1\le\infty$ and $0<q_0\neq q_1\le\infty$. Let $0<\theta<1$ and put
$$
\frac{1}{p}:=\frac{1-\theta}{p_0}+\frac{\theta}{p_1},\quad \frac{1}{q}:=\frac{1-\theta}{q_0}+\frac{\theta}{q_1}.
$$
If
\begin{equation}\label{central}
p\le q,
\end{equation}
then $T$ maps boundedly $L_p(\Omega)$ into $L_q(\Omega)$.
\end{thm}

One of the prominent applications of Marcinkiewicz interpolation theorem was given by Stein in his classical book \cite{S}.
The Hardy-Littlewood maximal operator is defined for every locally-integrable function $f$ on $\R^n$ by
$$
Mf(x)=\sup_{B\ni x}\frac{1}{|B|}\int_B |f(x)|dx,\quad x\in\R^n,
$$
where the supremum is taken over all balls in $\R^n$ containing $x$. It is easy to see that $M$ acts boundedly from $L_\infty(\R^n)$ into $L_\infty(\R^n)$.
Furthermore, one shows that $M$ maps $L_1(\R^n)$ into $L_{1,\infty}(\R^n)$. These two facts, combined with Theorem \ref{Marcin}, lead immediately to the boundedness
of $M$ on $L_p(\R^n)$ for every $1<p\le\infty$.

The study of the maximal operator in the frame of Lebesgue spaces with variable exponents attracted a lot of attention with the most important breakthroughs being achieved in \cite{Max1, Max2, Max3, Max4, CDF}.
It turned out, cf. \cite[Theorem 4.3.8]{DHHR}, that $M$ is bounded on $L_{p(\cdot)}(\R^n)$ if $p^->1$ with the function $1/p(\cdot)$ satisfying the so-called \emph{$\log$-H\"older continuity} conditions.
Under the same regularity conditions on $p$ it was proven in \cite{CDF}, that if $p^-\ge 1$  the maximal operator maps $L_{p(\cdot)}(\R^n)$ into $L_{p(\cdot),\infty}(\R^n)$.
Quite naturally, this raises the question if the boundedness of $M$ on Lebesgue spaces of variable integrability could be deduced from this weak-type estimate and some version of the Marcinkiewicz interpolation theorem.

In its abstract setting, the same question was already posed as an open problem in \cite{DHN}. 
We recall their notation first. We say, that the sublinear operator $T$ is of weak-type
$(\pi_0(\cdot),\pi_1(\cdot))$, if
$$
\lambda \|\chi_{\{x\in\R^n:|Tf(x)|>\lambda\}}\|_{\pi_1(\cdot)}\le c\|f\|_{\pi_0(\cdot)}%,\quad f\in L_{\pi_0(\cdot)}(\R^n), \lambda>0
$$
holds for some $c>0$ and all $f\in L_{\pi_0(\cdot)}(\R^n)$ and all $\lambda>0$. %independent of $L_{\pi_0(\cdot)}(\R^n)$ and $\lambda>0.$

Let $\pi_0(\cdot)$ and $\pi_1(\cdot)$ be two variable exponents and let $0<\theta<1$ be a real number. Then we put
$$
\frac{1}{\pi_\theta(x)}:=\frac{1-\theta}{\pi_0(x)}+\frac{\theta}{\pi_1(x)}.
$$

The Question 2.8 from \cite{DHN} becomes

{\bf Question 2.8}(\cite{DHN}, Marcinkiewicz Interpolation) Let $T$ be a sublinear operator that is of weak type
$(\pi_0(\cdot),\pi_0(\cdot))$ and $(\pi_1(\cdot),\pi_1(\cdot))$. Is $T$ then bounded from $L_{\pi_\theta(\cdot)}(\R^n)$ to $L_{\pi_\theta(\cdot)}(\R^n)$?

%Before we construct a counterexample, we recall Corollary 1.4.21 of \cite{G}.

We shall prove that the answer to Question 2.8 is negative. 

The basic idea of our construction is based on the observation that the condition 
\eqref{central} is necessary for the Marcinkiewicz interpolation theorem on usual $L_p(\Rn)$ with constant exponent, cf. \cite{H}.
It means that there are $0<p_0\not=p_1\le\infty$ and $0<q_0\not=q_1\le\infty$ and $0<\theta<1$ such that
$T$ is of weak type $(p_0,q_0)$ and $(p_1,q_1)$, $p>q$ and $T$ is not bounded from $L_p[0,1]$ to $L_q[0,1]$.

Then we set
$$
\tilde Tf(x):= \begin{cases} T(\chi_{[0,1]}f)(x-1), \quad x\in[1,2],\\
0 , \quad x\in[0,1).\end{cases}
$$
We put
$$
\pi_0(x):=\begin{cases}p_0,x\in[0,1),\\q_0,x\in[1,2]\end{cases}\quad\text{and}\quad
\pi_1(x):=\begin{cases}p_1,x\in[0,1),\\q_1,x\in[1,2].\end{cases}
$$
We obtain that $\tilde T$ is of weak type $(\pi_0(\cdot),\pi_0(\cdot))$ and of weak type $(\pi_1(\cdot),\pi_1(\cdot))$
but not of strong type $(\pi_\theta(\cdot),\pi_\theta(\cdot))$. Furthermore, a simple modification of this argument allows to construct 
a counterexample even for smooth parameters $\pi_0(\cdot)$ and $\pi_1(\cdot)$. To make the presentation self-contained, we provide a simple
construction of $T$ with the properties mentioned above.

%This (simple) method should be explored in the following sense
%\begin{itemize}
%\item How general can this construction be made? Is it true, that the Question 2.8 fails for all variable $\pi_0$ and/or variable $\pi_1$?
%\item One could define and study the basic properties of Lorentz spaces $L^{p(\cdot),q}$.
%\item Furthermore, using the $\ell_{q(\cdot)}(L_{p(\cdot)})$ construction of Almeida and H\"ast\"o, one could define $L^{p(\cdot), q(\cdot)}$
%with (hopefully) $L^{p(\cdot),p(\cdot)}=L^{p(\cdot)}$.
%\item For these spaces the interpolation theorem of Marcinkiewicz (i.e. 1.4.19 of Grafakos) could be true again!?
%\end{itemize}

\subsection{Specific counterexample for $T$}
In this section, we shall provide more details on the above given construction. Especially, we shall construct an operator $T$ which satisfies the assumptions from our counterexample.
Following the work of Hunt \cite{H,HW}, we define for $\alpha>0$ the following Hardy type operator
$$
(T_\alpha f)(x)=x^{-\alpha-1}\int_0^x f(t)dt, \quad 0<x<1.
$$
We observe first that $T_\alpha$ is linear and defined on all $L_1(0,1)$. Using the estimate $|\int_0^x f(t)dt|\le \int_0^x f^*(t)dt$ and H\"older's inequality
\begin{align*}
\|T_{1/2}f\|_{1,\infty}\le \sup_{0<x<1} x\cdot x^{-3/2}\int_0^x f^*(t)dt\le 
\sup_{0<x<1}x^{-1/2}\Bigl(\int_0^x (f^*(t))^2dt\Bigr)^{1/2}\cdot x^{1/2}=\|f\|_2,
\end{align*}
we obtain also the weak-type estimate $T_{1/2}:L_2(0,1)\to L_{1,\infty}(0,1)$. Furthermore, the boundedness of $T_{1/2}$ from $L_\infty(0,1)$ into $L_{2,\infty}(0,1)$
follows by
\begin{align*}
\|T_{1/2}f\|_{2,\infty}\le \sup_{0<x<1} x^{1/2}\cdot x^{-3/2}\int_0^x f^*(t)dt=\sup_{0<x<1}x^{-1}\cdot\int_0^x f^*(t)dt\le \|f\|_\infty.
\end{align*}
On the other hand, it is easy to see, that $T_{1/2}$ is not bounded from $L_4(0,1)$ into $L_{4/3}(0,1)$. Just take $f(t)=t^{-1/4}|\ln(t)|^{-1/4-\varepsilon}\chi_{[0,1/2]}(t)\in L_4(0,1)$ for $\varepsilon>0$
and calculate
\begin{align*}
\|T_{1/2}f\|_{4/3}&=\Bigl(\int_0^1 x^{-2} \Bigl(\int_0^x f(t)dt\Bigr)^{4/3}dx\Bigr)^{3/4}\\
&\ge \Bigl(\int_0^{1/2} x^{-2} \Bigl(\int_0^x t^{-1/4}|\ln(t)|^{-1/4-\varepsilon}dt\Bigr)^{4/3}dx\Bigr)^{3/4}\\
&\approx \Bigl(\int_0^{1/2} x^{-2} \Bigl(x^{3/4}|\ln(x)|^{-1/4-\varepsilon}\Bigr)^{4/3}dx\Bigr)^{3/4}\\
&=\Bigl(\int_0^{1/2}x^{-1}\cdot |\ln(x)|^{-1/3-\varepsilon\cdot 4/3}dx\Bigr)^{3/4}=\infty
\end{align*}
for $0<\varepsilon\le 1/2.$

\section{Open problems}
Although we presented some basic properties of the scale of Lorentz spaces with variable exponents, many questions remained opened for further investigations.
%In this section we comment on some open problems and further directions of investigation.\\

The first obvious generalization is to treat Lorentz spaces $L_{\p,\q}(\Omega,\mu)$ on arbitrary measure spaces $(\Omega,\mu)$
as it is usually done for Lorentz spaces with constant exponents, cf. \cite{Grafaneu}. We believe that the considerations above are also true in this case
and we have studied mainly spaces on $\R^n$ to simplify the notation.

It would be also highly desirable to obtain further results on real interpolation in this scale, complementing Theorem \ref{thm:interpol}.
Especially, it would be useful to have two variable exponent spaces as interpolation couple, i.e. to characterize $(L_{p_0(\cdot)}(\Rn),L_{p_1(\cdot)}(\Rn))_{\theta,q}$.

The use of $\ellqp$ spaces suggests yet another interesting idea, namely to allow for interpolation with variable parameter $\q$.
This option was already noticed in the introduction of \cite{AlHa10}. Unfortunately it turns out, and it is also mentioned in \cite{AlHa},
that the real interpolation spaces with variable parameter $\q$ lack in general the interpolation property.
On the other hand, it is possible, that under suitable conditions on the endpoint spaces, the real interpolation method with variable $\q$
works well and the interpolation property is restored again.
%This seems to be another obstacle regarding the spaces $\ellqp$ as it has been with the norm property in these spaces, see \cite{KV11}.\\

The last open problem is the starting point of this paper. If the Marcinkiewicz interpolation theorem would work in the scale of variable exponent spaces,
we would have found a very elegant way to prove the boundedness of the Hardy-Littlewood maximal operator on $\Lp$. Combining the weak estimate from \cite{CDF}
\begin{align*}
	M&:L_{p_0(\cdot)}(\Rn) \to L_{p_0(\cdot),\infty}(\Rn)\\
\intertext{with the trivial boundedness}
	M&:L_\infty(\Rn)\to L_\infty(\Rn)
\end{align*}
we could get the boundedness of $M$ on $\Lp$, with $\frac1\p=\frac{1-\theta}{p_0(\cdot)}$. Unfortunately the previous counterexample tells us, that Marcinkiewicz does not work on variable $\Lp$ spaces.
%We tried to get the boundedness on $\Lp$ directly using the weak $(p_0(\cdot),p_0(\cdot))$ and the strong $(\infty,\infty)$ estimate.
On the other hand, using Theorem \ref{thm:interpol}, we may easily show that
\begin{align*}
M:L_{\p,q}(\Rn)\to L_{\p,q}(\Rn)
\end{align*}
under the regularity assumptions on $\p$ as used in \cite{CDF}. This implies especially by chosing $q$ within $p^-\leq q\leq p^+$ the boundedness
\begin{align*}
M:L_{\p,p^-}(\Rn)\to L_{\p,p^+}(\Rn).
\end{align*}
This seems to be very suggestive as to why Marcinkiewicz interpolation might fail for the maximal operator. We would therefore be very much interested if it is possible to find additional conditions
on the sublinear operator $T$, which would ensure the validity of Marcinkiewicz interpolation.

\textbf{Acknowledgement:} The first author acknowledges the financial support by the DFG grants HA 2794/5-1 and KE 1847/1-1.
The second author acnowledges the support by the DFG Research Center MATHEON ``Mathematics for key technologies'' in Berlin.
Moreover we are very grateful for the comments of the anonymous referee which helped to improve our paper.

\end{document}